\documentclass[journal,onecolumn]{IEEEtran}

\usepackage{mathrsfs}
\usepackage{color}
\usepackage{enumerate}
\usepackage{mathrsfs}
\usepackage{amssymb}
\usepackage{amsfonts}
\usepackage{amsmath}
\usepackage{multirow}
\usepackage{amsthm}
\usepackage{url}

\newtheorem{theorem}{Theorem}[section]
\newtheorem{definition}[theorem]{Definition}
\newtheorem{lemma}[theorem]{Lemma}

\newtheorem{corollary}[theorem]{Corollary}

\newtheorem{problem}{Problem}

\begin{document}

\title{New results on sparse representations in unions of orthonormal bases}

\author{Tao Zhang and Gennian Ge
\thanks{The research of G. Ge was supported by the National Key Research and Development Program of China under Grant 2020YFA0712100, the National Natural Science Foundation of China under Grant 12231014, and Beijing Scholars Program.}
\thanks{T. Zhang is with the Institute of Mathematics and Interdisciplinary Sciences, Xidian University, Xi'an 710071, China (e-mail: zhant220@163.com).}
\thanks{G. Ge is with the School of Mathematics Sciences, Capital Normal University, Beijing 100048, China (e-mail: gnge@zju.edu.cn).}
}

\maketitle
\begin{abstract}
The problem of sparse representation has significant applications in signal processing. The spark of a dictionary plays a crucial role in the study of sparse representation. Donoho and Elad initially explored the spark, and they provided a general lower bound. When the dictionary is a union of several orthonormal bases, Gribonval and Nielsen presented an improved lower bound for spark. In this paper, we introduce a new construction of dictionary, achieving the spark bound given by Gribonval and Nielsen. Our result extends Shen et al.' s findings [IEEE Trans. Inform. Theory, vol. 68, pp. 4230--4243, 2022].
\end{abstract}

\section{Introduction}
Due to the applications in compressed sensing, researchers are interested in sparse representation problems \cite{BDE2009,EK2012,FR2013}. Given a real column vector $v\in\mathbb{R}^n$, the goal is to find an efficient representation of the signal $v$. Let $B=\{b_1,b_2,\dots,b_n\}$ be an orthonormal basis of $\mathbb{R}^n$, then
\begin{align*}
v=\begin{pmatrix}
    b_1 & b_2 & \cdots & b_n
  \end{pmatrix}\begin{pmatrix}
                 v_1 \\
                 v_2 \\
                 \vdots \\
                 v_n
               \end{pmatrix}=\sum_{i=1}^{n}v_ib_i,
\end{align*}
where $v_i=\langle v,b_i\rangle$. Now we consider a more general problem where the orthonormal basis is replaced by a dictionary of $\mathbb{R}^n$.
\begin{definition}
  A dictionary of $\mathbb{R}^n$ is a family of $N\ge n$ unit column vectors $\{d_i: i=1,2,\dots,N\}$ that spans $\mathbb{R}^n$. We will use the matrix notation $D=[d_1,d_2,\dots,d_N]$ for a dictionary.
\end{definition}
If $N>n$, then the representation of $v$ is not unique, i.e., there are many solutions for the equation $Dx=v$. The goal of compressed sensing is to find the sparest representation among them. Then it is natural to consider the following optimization problem:
\begin{align}\label{eq1}
\min_{x}\|x\|_{0}\text{ subject to }v=Dx.
\end{align}
In general, problem (\ref{eq1}) is NP hard \cite{N1995}. Following \cite{BDE2009,DE2003}, the following two concepts are important in the study of sparse representation.
\begin{definition}
  The mutual coherence $\mu(D)$ of a given matrix $D$ is the largest magnitude of the inner product between two columns of $D$, i.e.,
  \[\mu(D)=\max_{i\ne j}|\langle d_i,d_j\rangle|.\]
\end{definition}
\begin{definition}
  The spark $\eta(D)$ of a given matrix $D$ is the smallest number of columns from $D$ that are linearly dependent, i.e.,
  \[\eta(D)=\min_{x\in\ker(D),x\ne0}\|x\|_{0},\]
  where $\ker(D)$ is the null space of $D$.
\end{definition}
In \cite{DE2003,GR1997}, the authors proved that if the equation $v=Dx$ has a solution $x_0$ satisfying $\|x_0\|_{0}<\frac{\eta(D)}{2}$, then $x_0$ is the unique solution for problem (\ref{eq1}). Hence it is interesting to consider dictionary that supports the largest possible sparse vectors. The problem of computing the spark of a matrix is NP hard \cite{DE2003,TP2014}. In \cite{T2019}, the author gave a specialized algorithm to compute the spark. Dictionary with large spark also has applications in linear sensor arrays \cite{MP1998} and tensor decompositions \cite{KB2009}.

For any dictionary $D$, Donoho and Elad \cite{DE2003} proved that
\[\eta(D)\ge1+\frac{1}{\mu(D)}.\]
If $D$ is a union of two orthonormal bases, then a better bound can be found in \cite{EB2002}
\begin{align*}
\eta(D)\ge\frac{2}{\mu(D)}.
\end{align*}
The above bound was generalized by Gribonval and Nielsen \cite{GN2003}, they proved that if $D$ is a union of $q+1$ orthonormal bases, then
\begin{align}\label{eq2}
\eta(D)\ge(1+\frac{1}{q})\frac{1}{\mu(D)}.
\end{align}
In the same paper, the authors asked whether there exist examples for which the equality holds in (\ref{eq2}) when $q>1$. In \cite{SYSL2022}, the authors gave a positive answer to above question. Moreover, they proved that for $q=2^r$, $r$ is a positive integer, $t=1$ or 2, there exists a dictionary $D$ in $\mathbb{R}^{q^{2t}}$, which is a union of $q+1$ orthonormal bases, such that the spark of $D$ attains bound (\ref{eq2}). Then it is natural to ask the following problem.
\begin{problem}\label{problem}
For which $n,d$, there exists a dictionary $D$ in $\mathbb{R}^n$, which is a union of $d$ orthonormal bases, such that the spark of $D$ attains bound (\ref{eq2}).
\end{problem}
For the above problem, Shen et al. \cite{SYSL2022} settled the case $(n,d)=(q^{2t},q+1)$, where $q=2^r$, $r$ is a positive integer and $t=1$ or 2.
 We also refer the readers to \cite{DE2003,DH2001,GN2007,GN2003,SYSL2023} for more dictionaries that are unions of several orthonormal bases.

In this paper, we solve Problem~\ref{problem} for $(n,d)=(q^{2t},q+1)$, where $q=2^r$ and $t,r$ are any positive integers. More precisely, We prove that
\begin{theorem}\label{thm1}
  Let $t,r$ be positive integers, $q=2^r$, then there exists a dictionary $D$ in $\mathbb{R}^{q^{2t}}$ which is a union of $q+1$ orthonormal bases satisfying $\eta(D)=q^t+q^{t-1}$ and $\mu(D)=\frac{1}{q^t}$.
\end{theorem}
It is easy to see that the spark of dictionary in Theorem \ref{thm1} attains bound (\ref{eq2}). This paper is organized as follows. In Section \ref{pre}, we recall some basics of finite fields. In Section~\ref{proof}, we prove the main theorem.
\section{Preliminaries}\label{pre}
Let $m,n$ be integers with $m\mid n$. For any $a\in\mathbb{F}_{2^n}$, the trace of $a$ from $\mathbb{F}_{2^n}$ to $\mathbb{F}_{2^m}$ is defined by
\[\text{Tr}_{m}^{n}(a)=a+a^{2^m}+a^{2^{2m}}+\cdots+a^{2^{n-m}}.\]
It is easy to see that $\text{Tr}_{m}^{n}(a)\in \mathbb{F}_{2^m}$. We recall some properties of trace function.
\begin{lemma}\cite[Theorem 2.1.83]{M2013}\label{lemma3}
 Let $m,n$ be integers with $m\mid n$. Then the trace function satisfies the following properties:
 \begin{enumerate}
   \item $\text{Tr}_{m}^{n}(a+b)=\text{Tr}_{m}^{n}(a)+\text{Tr}_{m}^{n}(b)$ for all $a,b\in\mathbb{F}_{2^n}$;
   \item $\text{Tr}_{m}^{n}(ca)=c\text{Tr}_{m}^{n}(a)$ for all $a\in\mathbb{F}_{2^n}$ and $c\in\mathbb{F}_{2^m}$;
   \item $\text{Tr}_{m}^{n}(a^{2^m})=\text{Tr}_{m}^{n}(a)$ for all $a\in\mathbb{F}_{2^n}$;
   \item $|\{x\in\mathbb{F}_{2^n}:\ \text{Tr}_{m}^{n}(x)=c\}|=2^{n-m}$ for any $c\in\mathbb{F}_{2^m}$.
 \end{enumerate}
\end{lemma}

\begin{lemma}\cite[Theorem 2.26]{LN1997}
 Let $l,m,n$ be integers with $l\mid m\mid n$. Then $\text{Tr}_{l}^{n}(a)=\text{Tr}_{l}^{m}(\text{Tr}_{m}^{n}(a))$ for all $a\in\mathbb{F}_{2^n}$.
\end{lemma}
From Lemma~\ref{lemma3}, we have the following corollary.
\begin{corollary}\label{coro1}
$\sum_{x\in\mathbb{F}_{2^n}}(-1)^{\text{Tr}_{1}^{n}(x)}=0$.
\end{corollary}
\begin{proof}
  Since $|\{x\in\mathbb{F}_{2^n}:\ \text{Tr}_{1}^{n}(x)=0\}|=|\{x\in\mathbb{F}_{2^n}:\ \text{Tr}_{1}^{n}(x)=1\}|=2^{n-1}$, then $\sum_{x\in\mathbb{F}_{2^n}}(-1)^{\text{Tr}_{1}^{n}(x)}=0$.
\end{proof}
We also need the following lemma.
\begin{lemma}\cite[Corollary 3.79]{LN1997}\label{lemma1}
  For a positive integer $n$, the equation $x^2+ax+b=0$ with $a,b\in\mathbb{F}_{2^n}$, $a\ne0$, has solutions in $\mathbb{F}_{2^n}$ if and only if $\text{Tr}_{1}^{n}(\frac{b}{a^2})=0$.
\end{lemma}

\section{Proof of Theorem \ref{thm1}}\label{proof}
Let $t,r$ be positive integers, $m=tr$ and $n=2m=2tr$. Let $q=2^{r}$. Define $B_{q+1}=\{e_x: x\in\mathbb{F}_{2^n}\}$, and
\[B_{a}=\{B_{a,b}=\frac{1}{2^m}((-1)^{\text{Tr}_{1}^{m}(ax^{2^m+1})+\text{Tr}_{1}^{n}(bx)})_{x\in \mathbb{F}_{2^n}}: b\in\mathbb{F}_{2^n}\},\]
for $a\in\mathbb{F}_{2^r}$. We first prove that the vectors in $B_a$ form an orthonormal basis of $R^{2^n}$.

\begin{lemma}\label{lemma4}
  For each $a\in\mathbb{F}_{2^r}$, $B_a$ is an orthonormal basis of $R^{2^n}$.
\end{lemma}
\begin{proof}
  For any $b_1,b_2\in\mathbb{F}_{2^n}$ with $b_1\ne b_2$, we have
  \begin{align*}
  \langle B_{a,b_1},B_{a,b_2}\rangle=&\frac{1}{2^n}\sum_{x\in \mathbb{F}_{2^n}}(-1)^{\text{Tr}_{1}^{m}(ax^{2^m+1})+\text{Tr}_{1}^{n}(b_1x)}(-1)^{\text{Tr}_{1}^{m}(ax^{2^m+1})+\text{Tr}_{1}^{n}(b_2x)}\\
  =&\frac{1}{2^n}\sum_{x\in \mathbb{F}_{2^n}}(-1)^{\text{Tr}_{1}^{n}((b_1+b_2)x)}\\
  =&0,
  \end{align*}
  where the last equality follows from $b_1+b_2\ne0$, $\{(b_1+b_2)x: x\in \mathbb{F}_{2^n}\}=\{x: x\in \mathbb{F}_{2^n}\}$ and Corollary~\ref{coro1}.
  Note that $|B_a|=2^n$, then $B_a$ is an orthogonal basis of $R^{2^n}$.
\end{proof}

For any $a\in \mathbb{F}_{2^n}$, define $\overline{a}:=a^{2^m}$. The unit circle of $\mathbb{F}_{2^n}$ is the set
\[U=\{u\in\mathbb{F}_{2^n}: u^{2^m+1}=u\overline{u}=1\}.\]
Note that $\gcd(2^m+1,2^m-1)=1$, then for any $a\in \mathbb{F}_{2^n}^{*}$, there exists a unique representation $a=uv$, where $u\in U$ and $v\in \mathbb{F}_{2^m}^{*}$. Now we consider the inner product between vectors from $B_{a_1}$ and $B_{a_2}$, where $a_1\ne a_2$.

\begin{lemma}\label{lemma5}
  For any $a_1,a_2\in\mathbb{F}_{2^r}$ with $a_1\ne a_2$, and $b_1,b_2\in\mathbb{F}_{2^n}$, we have $|\langle B_{a_1,b_1},B_{a_2,b_2}\rangle|=\frac{1}{2^m}$.
\end{lemma}

\begin{proof}
  \begin{align*}
  \langle B_{a_1,b_1},B_{a_2,b_2}\rangle=&\frac{1}{2^n}\sum_{x\in \mathbb{F}_{2^n}}(-1)^{\text{Tr}_{1}^{m}(a_1x^{2^m+1})+\text{Tr}_{1}^{n}(b_1x)}(-1)^{\text{Tr}_{1}^{m}(a_2x^{2^m+1})+\text{Tr}_{1}^{n}(b_2x)}\\
  =&\frac{1}{2^n}\sum_{x\in \mathbb{F}_{2^n}}(-1)^{\text{Tr}_{1}^{m}((a_1+a_2)x^{2^m+1})+\text{Tr}_{1}^{n}((b_1+b_2)x)}.
  \end{align*}
  Set $a=a_1+a_2$, $b=b_1+b_2$, then $a\ne0$ and
  \begin{align*}
  \langle B_{a_1,b_1},B_{a_2,b_2}\rangle=&\frac{1}{2^n}\sum_{x\in \mathbb{F}_{2^n}}(-1)^{\text{Tr}_{1}^{m}(ax^{2^m+1})+\text{Tr}_{1}^{n}(bx)}\\
  =&\frac{1}{2^n}\sum_{x\in \mathbb{F}_{2^n}}(-1)^{\text{Tr}_{1}^{m}(ax^{2^m+1}+bx+(bx)^{2^m})}\\
  =&\frac{1}{2^n}\sum_{x\in \mathbb{F}_{2^n}}(-1)^{\text{Tr}_{1}^{m}(ax\overline{x}+bx+\overline{bx})}\\
  =&\frac{1}{2^n}(1+\sum_{u\in U}\sum_{v\in\mathbb{F}_{2^m}^{*}}(-1)^{\text{Tr}_{1}^{m}(auv\overline{uv}+buv+\overline{buv})})\\
  =&\frac{1}{2^n}(1+\sum_{u\in U}\sum_{v\in\mathbb{F}_{2^m}^{*}}(-1)^{\text{Tr}_{1}^{m}(av^2+buv+\overline{bu}v)})\\
  =&\frac{1}{2^n}(1+\sum_{u\in U}\sum_{v\in\mathbb{F}_{2^m}^{*}}(-1)^{\text{Tr}_{1}^{m}((a'+bu+\overline{bu})v)}),
  \end{align*}
  where $a'=a^{2^{m-1}}$. Let $b=b_1b_2$, where $b_1\in U$ and $b_2\in \mathbb{F}_{2^m}^{*}$, then $a'+bu+\overline{bu}=a'+b_2(b_1u+\overline{b_1u})$.
  Let
  \[N=|\{x\in U: a'+b_2(x+\overline{x})=0\}|.\]
  If $u\in U$ is a solution for equation $a'+b_2(x+\overline{x})=0$, then $a'+b_2(u+\overline{u})=0$. We have $a'u+b_2(u^2+1)=0$, i.e., $b_2u^2+a'u+b_2=0$. Hence $N\le2$. On the other hand, if $u\in U$ is a solution for equation $a'+b_2(x+\overline{x})=0$, then $\overline{u}\in U$ is also a solution for equation $a'+b_2(x+\overline{x})=0$. If $u=\overline{u}$, then $u^2=u\overline{u}=1$. Thus $u=1$, this leads to $a'=0$, which is a contradiction. Therefore $N=0$ or 2. Then we have
  \[\langle B_{a_1,b_1},B_{a_2,b_2}\rangle=\begin{cases}-\frac{1}{2^m},&\textup{ if } N=0;\\
\frac{1}{2^m},&\textup{ if }N=2.\end{cases}\]
This finishes the proof.
\end{proof}

The inner product between vectors from $B_a$ and $B_{q+1}$ is easy to get.
\begin{lemma}\label{lemma8}
  For any $a\in\mathbb{F}_{2^r}$, $b,x\in\mathbb{F}_{2^n}$, we have $|\langle B_{a,b},e_x\rangle|=\frac{1}{2^m}$.
\end{lemma}
\begin{proof}
 Since
  $\langle B_{a,b},e_x\rangle=\frac{1}{2^m}(-1)^{\text{Tr}_{1}^{m}(ax^{2^m+1})+\text{Tr}_{1}^{n}(bx)}$, then $|\langle B_{a,b},e_x\rangle|=\frac{1}{2^m}$.
\end{proof}

Now we consider the linear dependency of vectors in $B_a$ and $B_{q+1}$, we first prove a lemma.
\begin{lemma}\label{lemma2}
  For any $y\in\mathbb{F}_{2^m}$, we have
    \[|\{b\in\mathbb{F}_{2^m}: \text{Tr}_{1}^{m}(by)=0,\text{Tr}_{r}^{m}(b)=0\}|=\begin{cases}2^{m-r-1},&\textup{ if } y\notin\mathbb{F}_{2^{r}};\\
2^{m-r},&\textup{ if }y\in\mathbb{F}_{2^{r}}.\end{cases}\]
\end{lemma}
\begin{proof}
  We can compute to get
  \begin{align*}
  &|\{b\in\mathbb{F}_{2^m}: \text{Tr}_{1}^{m}(by)=0,\text{Tr}_{r}^{m}(b)=0\}|\\
  =&\frac{1}{2^{r+1}}\sum_{b\in\mathbb{F}_{2^m}}(1+(-1)^{\text{Tr}_{1}^{m}(by)})\sum_{\alpha\in\mathbb{F}_{2^r}}(-1)^{\text{Tr}_{1}^{r}(\alpha\text{Tr}_{r}^{m}(b))}\\
  =&\frac{1}{2^{r+1}}\sum_{\alpha\in\mathbb{F}_{2^r}}\sum_{b\in\mathbb{F}_{2^m}}(1+(-1)^{\text{Tr}_{1}^{m}(by)})(-1)^{\text{Tr}_{1}^{m}(\alpha b)}\\
  =&\frac{1}{2^{r+1}}\sum_{\alpha\in\mathbb{F}_{2^r}}\sum_{b\in\mathbb{F}_{2^m}}((-1)^{\text{Tr}_{1}^{m}(\alpha b)}+(-1)^{\text{Tr}_{1}^{m}(\alpha b+by)})\\
  =&2^{m-r-1}+\frac{1}{2^{r+1}}\sum_{\alpha\in\mathbb{F}_{2^r}}\sum_{b\in\mathbb{F}_{2^m}}((-1)^{\text{Tr}_{1}^{m}((\alpha+y)b)})\\
  =&\begin{cases}2^{m-r-1},&\textup{ if } y\notin\mathbb{F}_{2^{r}};\\
2^{m-r},&\textup{ if }y\in\mathbb{F}_{2^{r}}.\end{cases}
  \end{align*}
\end{proof}

Define
\[S=\{x\in\mathbb{F}_{2^n}:\ \text{Tr}_{r}^{m}(x^{2^m+1})=0, x+x^{2^m}\in\mathbb{F}_{2^r}\}.\]
Then we find that the following vectors are linearly dependent.
\begin{lemma}\label{lemma6}
  $\sum_{a\in\mathbb{F}_{2^r}}\sum_{b\in\mathbb{F}_{2^m},\text{Tr}_{r}^{m}(b)=0}B_{a,b}-\sum_{x\in S}e_{x}=0$.
\end{lemma}
\begin{proof}
  We can compute to get
  \begin{align*}
  &\sum_{a\in\mathbb{F}_{2^r}}(-1)^{Tr_{1}^{m}(ax^{2^m+1})}\\
  =&\sum_{a\in\mathbb{F}_{2^r}}(-1)^{Tr_{1}^{r}(a\text{Tr}_{r}^{m}(x^{2^m+1}))}\\
  =&\begin{cases}0,&\textup{ if } \text{Tr}_{r}^{m}(x^{2^m+1})\ne0;\\
2^{r},&\textup{ if }\text{Tr}_{r}^{m}(x^{2^m+1})=0,\end{cases}
  \end{align*}
  and
  \begin{align*}
  &\sum_{b\in\mathbb{F}_{2^m},\text{Tr}_{r}^{m}(b)=0}(-1)^{\text{Tr}_{1}^{n}(bx)}\\
  =&\sum_{b\in\mathbb{F}_{2^m},\text{Tr}_{r}^{m}(b)=0}(-1)^{\text{Tr}_{1}^{m}(b(x+x^{2^m}))}\\
  =&\begin{cases}0,&\textup{ if } x+x^{2^m}\notin\mathbb{F}_{2^{r}};\\
2^{m-r},&\textup{ if }x+x^{2^m}\in\mathbb{F}_{2^{r}},\end{cases}
  \end{align*}
  where the last equation follows from Lemma~\ref{lemma2}. For a vector $A\in\mathbb{F}_{2^n}^{2^n}$, we use $[A]_x$ to denote its $x$-th coordinate. From above two equations, we have
  \begin{align*}
  &[\sum_{a\in\mathbb{F}_{2^r}}\sum_{b\in\mathbb{F}_{2^m},\text{Tr}_{r}^{m}(b)=0}B_{a,b}]_x\\
  =&\sum_{a\in\mathbb{F}_{2^r}}\sum_{b\in\mathbb{F}_{2^m},\text{Tr}_{r}^{m}(b)=0}\frac{1}{2^m}(-1)^{\text{Tr}_{1}^{m}(ax^{2^m+1})+\text{Tr}_{1}^{n}(bx)}\\
  =&\frac{1}{2^m}\sum_{a\in\mathbb{F}_{2^r}}(-1)^{\text{Tr}_{1}^{m}(ax^{2^m+1})}\sum_{b\in\mathbb{F}_{2^m},\text{Tr}_{r}^{m}(b)=0}(-1)^{\text{Tr}_{1}^{n}(bx)}\\
  =&\begin{cases}0,&\textup{ if }\text{Tr}_{r}^{m}(x^{2^m+1})\ne0 \text{ or } x+x^{2^m}\notin\mathbb{F}_{2^{r}};\\
1,&\textup{ if }\text{Tr}_{r}^{m}(x^{2^m+1})=0, x+x^{2^m}\in\mathbb{F}_{2^{r}},\end{cases}\\
=&\begin{cases}0,&\textup{ if }x\notin S;\\
1,&\textup{ if }x\in S.\end{cases}
  \end{align*}
  Hence $\sum_{a\in\mathbb{F}_{2^r}}\sum_{b\in\mathbb{F}_{2^m},\text{Tr}_{r}^{m}(b)=0}B_{a,b}-\sum_{x\in S}e_{x}=0$.
\end{proof}
Now we compute the size of $S$.
\begin{lemma}\label{lemma7}
  $|S|=2^{m-r}.$
\end{lemma}
\begin{proof}
 For any $x\in S$, let $A=x\overline{x}$ and $B=x+\overline{x}$. Then $\text{Tr}_{r}^{m}(A)=0$ and $B\in\mathbb{F}_{2^r}$.

 If $x\in\mathbb{F}_{2^m}$, then $B=x+\overline{x}=0\in \mathbb{F}_{2^r}$ and $0=\text{Tr}_{r}^{m}(A)=\text{Tr}_{r}^{m}(x^2)=\text{Tr}_{r}^{m}(x)$. By Lemma~\ref{lemma3}, there are $2^{m-r}$ such $x$.

 If $x\in\mathbb{F}_{2^n}\backslash\mathbb{F}_{2^{m}}$, then $x,\overline{x}$ are solutions of equation $X^2+BX+A=0$. Note that $B,A\in\mathbb{F}_{2^m}$, by Lemma~\ref{lemma1}, $x,\overline{x}\in\mathbb{F}_{2^n}\backslash\mathbb{F}_{2^{m}}$ if and only if $\text{Tr}_{1}^{m}(\frac{A}{B^2})=1$. On the other hand, $\text{Tr}_{1}^{m}(\frac{A}{B^2})=\text{Tr}_{1}^{r}(\text{Tr}_{r}^{m}(\frac{A}{B^2}))=\text{Tr}_{1}^{r}(\frac{1}{B^2}\text{Tr}_{r}^{m}(A))=0$, which is a contradiction. Therefore, $|S|=2^{m-r}.$
\end{proof}
{\noindent  \it Proof of Theorem \ref{thm1}:  }
Let $D=B_{q+1}\cup(\cup_{a\in\mathbb{F}_{2^r}}B_a)$, then $D$ is a dictionary in $\mathbb{R}^{q^{2t}}$. By Lemma~\ref{lemma4}, $D$ is a union of $2^r+1=q+1$ orthonormal bases. By Lemmas~\ref{lemma5} and \ref{lemma8}, we have $\mu(D)=\frac{1}{2^m}=\frac{1}{q^t}$. From Lemma~\ref{lemma6}, the vectors in the set
\[(\cup_{a\in\mathbb{F}_{2^r}}\cup_{b\in\mathbb{F}_{2^m},\text{Tr}_{r}^{m}(b)=0}B_{a,b})\cup(\cup_{x\in S}e_{x})\]
are linearly dependent. By Lemma~\ref{lemma7}, we can get
\begin{align*}
|(\cup_{a\in\mathbb{F}_{2^r}}\cup_{b\in\mathbb{F}_{2^m},\text{Tr}_{r}^{m}(b)=0}B_{a,b})\cup(\cup_{x\in S}e_{x})|=&2^r\cdot2^{m-r}+2^{m-r}\\
=&2^m+2^{m-r}\\
=&q^t+q^{t-1}.
\end{align*}
Hence $\eta(D)=q^t+q^{t-1}$. This finishes the proof.

\end{document}